\documentclass[11pt, a4paper]{article}
\usepackage{amssymb,amsmath,amsthm,tocloft}
\usepackage{fancyhdr}
\usepackage{comment}
\usepackage[british]{babel}
\usepackage{verbatim}
\usepackage{stackengine}
\usepackage[usenames,dvipsnames,table]{xcolor}
\usepackage{graphicx,tikz,caption,subcaption}
\usetikzlibrary{decorations.pathreplacing}
\usetikzlibrary{decorations.pathmorphing}
\usetikzlibrary{calc,arrows,decorations.markings}
\tikzset{snake it/.style={decorate, decoration=snake}}
\usepackage{bbm,dsfont}
\usepackage[d]{esvect}

\usepackage{hyperref}

\usepackage{todonotes} % % % % % % % % % % % % % % %
\usepackage[margin=20mm,bottom=15mm,footskip=7mm,top=15mm]{geometry}

\let\OLDthebibliography\thebibliography
\renewcommand\thebibliography[1]{
  \OLDthebibliography{#1}
  \setlength{\parskip}{0pt}
  \setlength{\itemsep}{0pt plus 0.3ex}
}

\newtheorem{theorem}{Theorem}[section]
\newtheorem{lemma}[theorem]{Lemma}

\newcommand{\repeatlabel}{}
\newtheorem*{repeatlemma}{Lemma \repeatlabel}
\theoremstyle{definition}

\newcommand{\bR}{\mathbb{R}}

\newcommand{\cF}{\mathcal{F}}

\newcommand{\size}[1]{\left| #1 \right|}
\newcommand{\bfm}{\mathbf m}
\newcommand{\bfn}{\mathbf n}

\newcommand{\bfr}{\mathbf r}

\newcommand{\w}{\vv{\mathrm{wsat}}}
\newcommand{\wsat}{\mathrm{wsat}}

\makeatletter
\providecommand*{\cupdot}{%
  \mathbin{%
    \mathpalette\@cupdot{}%
  }%
}
\newcommand*{\@cupdot}[2]{%
  \ooalign{%
    $\m@th#1\cup$\cr
    \hidewidth$\m@th#1\cdot$\hidewidth
  }%
}
\makeatother

\usepackage{enumitem}

\newlist{thmlist}{enumerate}{1}
\setlist[thmlist]{label=(\arabic{thmlisti}), ref=\thethm.(\arabic{thmlisti}),noitemsep}

\usepackage{thmtools}
\theoremstyle{plain}
\declaretheorem[
    name=Theorem,
    Refname={Theorem,Theorems},
    numberwithin=section]{thm}

\declaretheorem[
    name=Proposition,
    Refname={Proposition,Propositions},
    sibling=thm]{prop}
\usepackage{nameref,hyperref}
\usepackage[capitalize]{cleveref}

\Crefname{thm}{Theorem}{Theorems}
\Crefname{lem}{Lemma}{Lemmas}
\Crefname{prop}{Proposition}{Propositions}

\DeclareMathOperator{\spn}{span}

\title{Colorful fractional Helly theorem via weak saturation} 
\author{
Debsoumya Chakraborti\thanks{Mathematics Institute, University of Warwick, Coventry, United Kingdom, \texttt{debsoumya.chakraborti@warwick.ac.uk}. D.~Chakraborti was supported by the European Research Council (ERC) under the European Union Horizon 2020 research and innovation programme (grant agreement No. 947978).}
\and
Minho Cho\thanks{Extremal Combinatorics and Probability Group, Institute for Basic Science, Daejeon, Republic of Korea, \texttt{minhocho@ibs.re.kr}. M.~Cho was supported by the Institute for Basic Science (IBS-R029-C4).} 
\and 
Jinha Kim\thanks{Department of Mathematics, Chonnam National University, Gwangju, Republic of Korea, \texttt{jinhakim@jnu.ac.kr}.}~\thanks{Discrete Mathematics Group, Institute for Basic Science, Daejeon, Republic of Korea. J.~Kim was supported by the Institute for Basic Science (IBS-R029-C1).} 
\and 
Minki Kim\thanks{Division of Liberal Arts and Sciences, GIST, Gwangju, Republic of Korea, \texttt{minkikim@gist.ac.kr}. M.~Kim was supported by Basic Science Research Program through the National Research Foundation of Korea (NRF) funded by the Ministry of Education (NRF-2022R1F1A1063424) and by GIST Research Project grant funded by the GIST in 2024.}
}

\begin{document}

\maketitle

\begin{abstract}
Two celebrated extensions of the classical Helly's theorem are the fractional Helly theorem and the colorful Helly theorem. 
Bulavka, Goodarzi, and Tancer recently established the optimal bound for the unified generalization of the fractional and the colorful Helly theorems using a colored extension of the exterior algebra. 
In this paper, we combinatorially reduce both the fractional Helly theorem and its colorful version to a classical problem in extremal combinatorics known as {weak saturation}. 
No such results connecting the fractional Helly theorem and weak saturation are known in the long history of literature. 
These reductions, along with basic linear algebraic arguments for the reduced weak saturation problems, let us give new short proofs of the optimal bounds for both the fractional Helly theorem and its colorful version without using exterior algebra. 
\end{abstract}

\section{Introduction}
We explore the intriguing connection between two classical problems, one from discrete geometry and the other from extremal combinatorics. 
The problem in discrete geometry is the fractional Helly theorem which asks for the maximal size of an intersecting subfamily in a given finite family of convex sets in $\mathbb{R}^d$ when the density of intersecting $(d+1)$-tuples is positive.
The problem in extremal combinatorics for a given (hyper)graph $F$, called \emph{weak saturation of $F$}, asks for the minimum number of (hyper)edges in an $n$-vertex (hyper)graph $H$ such that the non-edges of $H$ can be added one by one such that a new copy of $F$ is created at every step.
In the 80s, a powerful algebraic technique known as \emph{exterior algebra} was developed to resolve both the fractional Helly theorem and the weak saturation problem for complete (hyper)graphs. 
Since then, exterior algebra has been influential and used numerous times for problems of similar nature. 
For both problems, proving tight bounds withstands combinatorial methods and requires algebraic or geometric arguments. Thus, it has been an open question whether purely combinatorial solutions exist for these problems. 
Although similar algebraic approaches work for both problems, to the best of our knowledge, there have been no results relating these two problems. 
We bridge this gap by reducing the fractional Helly problem to a weak saturation problem in a purely combinatorial way. 
Before diving into our results, we briefly outline some more history of these problems.

\subsection{Colorful and fractional Helly theorems}
Helly's theorem, now a century old, asserts that a finite number of convex sets in $\mathbb{R}^d$ have a point in common if and only if every $d+1$ or fewer of them have a point in common.
Inspired by the features of this theorem, results that derive global intersection properties from local intersection patterns are called Helly-type theorems.
For an overview of such problems and results, see, for example, \cite{ALS17, BK22, Tan13}.

Two of the most important Helly-type theorems from a combinatorial perspective are the fractional Helly theorem and the colorful Helly theorem.
Here, a family of nonempty sets is called \emph{intersecting} if all its members have a point in common.
\begin{thm}[Fractional Helly theorem, \cite{Kalai84}]\label{thm:frachel}
    Let $F$ be a finite family of convex sets in $\mathbb{R}^d$, and let $r$ be a nonnegative integer such that $d+r \leq \size{F}$.
If $F$ contains no intersecting subfamily of size $d+r+1$, then the number of intersecting $(d+1)$-tuples in $F$ is at most $\binom{\size{F}}{d+1}-\binom{\size{F}-r}{d+1}$.
\end{thm}
\begin{thm}[Colorful Helly theorem, \cite{Ba}]\label{thm:colhel}
    Let $F_1,\ldots,F_{d+1}$ be finite families of convex sets in~$\mathbb{R}^d$.
    Suppose that for every choice $A_i \in F_i$ with $i\in [d+1]$, the intersection $\bigcap_{i=1}^{d+1}A_i$ is nonempty.
    Then, there exists $i\in [d+1]$ such that $F_i$ is intersecting.
\end{thm}

The proof of Theorem~\ref{thm:frachel} by Kalai uses the exterior algebra. Alternative proofs are presented in \cite{AlonKalai,EckhoffUBT}. In particular, the proof in \cite{AlonKalai} uses the Bollob\'as set-pair theorem, which allows us to prove Theorem~\ref{thm:frachel} using linear algebra instead of exterior algebra.
All proofs are based on Wegner's observation~\cite{Wegner} that for a finite family $F$ of convex sets in $\mathbb{R}^d$, its {\em nerve}, which is a simplicial complex on $F$ whose simplices are exactly the intersecting subfamilies of $F$, is $d$-collapsible.

Given an abstract simplicial complex $K$ on a finite vertex set, an {\em elementary $d$-collapse} in $K$ is a process of taking a face $\sigma$ of size at most $d$ that is contained in a unique maximal face and deleting all faces containing $\sigma$.
A simplicial complex is {\em $d$-collapsible} if one can remove all faces by a finite sequence of elementary $d$-collapses.
With the notion of $d$-collapsibility, one can reformulate the fractional and the colorful Helly theorems as follows.

\begin{thm}\label{thm:helly_d-collapsible}
    Let $K$ be a $d$-collapsible complex on $V$ with $\size{V} = n$.
    \begin{thmlist}
        \item\label[theorem]{thm:fracHelly} {\normalfont (Fractional Helly theorem)} If $\dim(K) \leq d+r-1$, then the number of $d$-dimensional faces of $K$ is at most $\binom{n}{d+1}-\binom{n-r}{d+1}$. 
        {\normalfont (Recall that $\dim\sigma:=\size{\sigma}-1$ and $\dim(K):=\max_{\sigma \in K}\dim\sigma$.)}
        \item\label[theorem]{thm:colorfracHelly} {\normalfont (Colorful Helly theorem)} Consider a partition $V = V_1 \cup  \cdots\cup  V_{d+1}$. If all rainbow sets {\normalfont (the sets containing at most one element from each $V_i$)} of size $d+1$ are faces of $K$, then some $V_i$ is a face of~$K$.
    \end{thmlist}
\end{thm}

A unified generalization of the fractional and the colorful Helly theorem was first qualitatively presented by B\'ar\'any, Fodor, Montejano, Oliveros, and P\'{o}r~\cite{BaranyCFHelly}, and was given robustness by Kim~\cite{Kim}. More recently, the quantitatively tight version, conjectured by Kim~\cite{Kim}, was proved by Bulavka, Goodarzi, and Tancer~\cite{OptimalCFHelly}. Here is the statement of optimal colorful fractional Helly theorem for $d$-collapsible complexes.
For optimal constructions, see~\cite{OptimalCFHelly,Kim}.

\begin{thm}[Colorful fractional Helly theorem, \cite{OptimalCFHelly}] \label{thm:colfrachel}
Let $V$ be a set with a partition $V = V_1 \cup \cdots\cup V_{d+1}$ such that $\size{V_i} = n_i\ge r_i$ for every $i\in [d+1]$.
If $K$ is a $d$-collapsible simplicial complex on~$V$ with $\dim(K[V_i]) \le r_i - 1$ for every $i\in [d+1]$, then the number of rainbow faces of size $d+1$ is at most $\prod_{i} n_i - \prod_{i} (n_i - r_i)$.
\end{thm}

In order to prove Theorem~\ref{thm:colfrachel}, Bulavka, Goodarzi, and Tancer introduced the colored version of exterior algebra, and they actually gave an optimal upper bound on the number of faces $\sigma$ in $K$ with $\size{\sigma \cap V_i} = k_i$ for given nonnegative integers $k_1, \ldots, k_{d+1}$.

\subsection{Weak saturation and related problems}

Let $F$, $G$, and $H$ be hypergraphs such that $H\subseteq G$ (i.e., $H$ is a subhypergraph of $G$).
We say that $H$ is \emph{weakly $F$-saturated in $G$} if there exists an ordering $e_1, \ldots, e_m$ of the edges in $E(G) \setminus E(H)$ so that for every $k \in [m]$, the hypergraph obtained from $H$ after adding the edges $e_1, \ldots, e_k$, denote by $H \cup \{e_1, \ldots, e_k\}$, has a copy of $F$ containing $e_k$.
We call $G$ the \emph{host (hyper)graph} for the weak saturation.
We omit the host graph whenever it is clear from the context.
The sequence of edges $e_1, \ldots, e_m$ is called a \emph{weak $F$-saturation sequence} of $H$.
The minimum number of edges in a hypergraph $H$ that is weakly $F$-saturated in $G$ is called the \emph{weak $F$-saturation number} in $G$ and is denoted by $\wsat(G; F)$.

We next consider the \emph{directed} variant of weak saturation.
We call an $r$-uniform hypergraph simply an \emph{$r$-graph}.
For $r$-partite $r$-graphs $F$ and $H$, and with a given fixed partition $(U_1, \ldots, U_r)$ of $V(F)$ and $(V_1, \ldots, V_r)$ of $V(H)$, a copy of $F$ in $H$ is called \emph{directed} if the corresponding embedding ${\varphi:V(F) \rightarrow V(H)}$ satisfies $\varphi(U_i) \subseteq V_i$ for all $i \in [r]$.
Let $G$ and $H$ be $r$-partite $r$-graphs with a given vertex partition $(V_1, \ldots, V_r)$ such that $H\subseteq G$.
Let $F$ be another $r$-partite $r$-graph with a given vertex partition $(U_1, \ldots, U_r)$.
Now, a weak $F$-saturation sequence $e_1, \ldots, e_m$ of~$H$ is called \emph{directed} if for every $k \in [m]$, there is a directed copy of $F$ using $e_k$ in $H \cup \{e_1, \ldots, e_k\}$. 
Similar to before, the minimum number of edges in a directed weakly $F$-saturated hypergraph $H$ is called the \emph{directed weak $F$-saturation number} in $G$, and we denote this number by~$\w(G; F)$.

Weak saturation was first introduced by Bollob\'as~\cite{WeakSaturation}.
In the most natural case when both $G$ and $F$ are complete $r$-graphs, the value of $\wsat(G; F)$ is determined independently by Frankl~\cite{Frankl} and Kalai~\cite{Kalai81, Kalai85}.
For $G$ and $F$ both complete $r$-partite $r$-graphs, Alon~\cite{Alon} determined $\w(G; F)$ and
using this result, Moshkovitz and Shapira~\cite{HypergraphSaturation} later determined $\wsat(G; F)$.
Recently, this has been generalized to all complete $q$-partite $r$-graphs $G$ and $F$ with $q \geq r$ by Bulavka, Tancer, and Tyomkyn~\cite{WeakSaturationHypergraph} using the colored exterior algebra mentioned earlier.
Not many values of $\w(G; F)$ and $\wsat(G; F)$ are known for other hypergraphs $G$ and $F$; see, e.g.,~\cite{Bootstrap, MorrisonNoelScott, Pikhurko, ShapiraTyomkyn} for some results.

Now consider a family $\cF = \{F_1, \ldots, F_s\}$ of $r$-graphs.
We say that $H$ is \emph{weakly $\cF$-saturated} in $G$ if there exists an ordering $e_1, \ldots, e_m$ of the edges in $E(G) \setminus E(H)$ so that for every $k \in [m]$, the hypergraph $H \cup \{e_1, \ldots, e_k\}$ contains a copy of some $F_j \in \cF$ using $e_k$.
The weak $\cF$-saturation number in $G$ is defined in a similar way and is denoted by $\wsat(G; \cF)$.
\emph{Directed weak $\cF$-saturation} and the \emph{directed weak $\cF$-saturation number} $\w(G; \cF)$ are defined in the analogous way.

\subsection{Main result}
We point out that the fractional Helly theorem and the colorful fractional Helly theorem actually can be interpreted as a (directed) weak saturation of a family of complete multipartite hypergraphs.
We denote by $K^{(r)}_n$ the complete $r$-graph on $n$ vertices.
For a tuple $\bfn = (n_1, \ldots, n_r)$, we denote by $K^{(r)}_\bfn$ the complete $r$-partite $r$-graph on vertex partition $(V_1, \ldots, V_r)$ with $\size{V_i} = n_i$ for every $i \in [r]$.

\begin{thm}[Reduction of (colorful) fractional Helly to weak saturation] \label{lem:CollapseSaturate}
\mbox{}
\begin{thmlist}
\item \label[theorem]{thm:weaksat}
Let $K$ be a $d$-collapsible simplicial complex on $V$ as in \Cref{thm:helly_d-collapsible}.
Let $H \subseteq K^{(d+1)}_n$ be the $(d+1)$-graph on $V$ whose edges are exactly the $d$-dimensional non-faces of $K$.
Then, $H$ is weakly $K^{(d+1)}_\bfm$-saturated, where $\bfm$ is the $(d+1)$-tuple $(1,\dots,1,n-d-r+1)$.

\item \label[theorem]{thm:weaksatfamily}
Let $K$ and $V = V_1 \cup \cdots \cup V_{d+1}$ be as in \Cref{thm:colfrachel}.
Let $H \subseteq K^{(d+1)}_\bfn$ be the $(d+1)$-partite $(d+1)$-graph whose edges are exactly the rainbow non-faces of size $d+1$ in $K$.
For every $i \in [d+1]$, let $\bfm_i = (1, \ldots, n_i - r_i + 1, \ldots, 1)$ be the $(d+1)$-tuple of integers whose $i$-th entry is $n_i - r_i + 1$ and all other entries are 1.
Then, $H$ is directed weakly $\{K^{(d+1)}_{\bfm_1}, K^{(d+1)}_{\bfm_2}, \ldots, K^{(d+1)}_{\bfm_{d+1}}\}$-saturated.
\end{thmlist}
\end{thm}

The weak saturation number for each situation is given as follows.
\begin{prop}[Weak saturation numbers] \label{thm:weak sat} 
\mbox{}
\begin{thmlist}
\item \label[proposition]{thm:weaksatnumber}  
Let $\bfm$ be as in~\Cref{thm:weaksat}.
Then, $\wsat \left(K^{(d+1)}_n; K^{(d+1)}_\bfm\right) = \binom{n - r}{d + 1}$.

\item \label[proposition]{thm:weaksatfamilynumber} 
Let $\bfm_i$ be as in~\Cref{thm:weaksatfamily}.
Then, $\w \left(K^{(d+1)}_\bfn; \{K^{(d+1)}_{\bfm_1}, \ldots, K^{(d+1)}_{\bfm_{d+1}}\}\right) = \prod_{i\in [d+1]} (n_i - r_i)$.
\end{thmlist}
\end{prop}

Note that \Cref{thm:frachel,thm:colfrachel} follow from \Cref{lem:CollapseSaturate,thm:weak sat}.
The equality in \Cref{thm:weaksatnumber} is achieved by the clique $K^{(d+1)}_{n-r}$, whereas the equality in \Cref{thm:weaksatfamilynumber} is achieved by the complete $(d+1)$-partite $(d+1)$-graph $K^{(d+1)}_{\bfn - \bfr}$, where ${\bfn - \bfr = (n_1-r_1,\dots,n_{d+1}-r_{d+1})}$.
Although both weak saturation numbers in \Cref{thm:weak sat} are known in the literature (see \cite[Theorem~1.1]{Pikhurko} for \Cref{thm:weaksatnumber} and \cite[Theorem~4]{Bootstrap} for \Cref{thm:weaksatfamilynumber}), we include short proofs using only simple linear algebra. 
The whole story is summarized in the following diagram.
\begin{figure}[htbp]
    \centering
    \includegraphics[scale=.82]{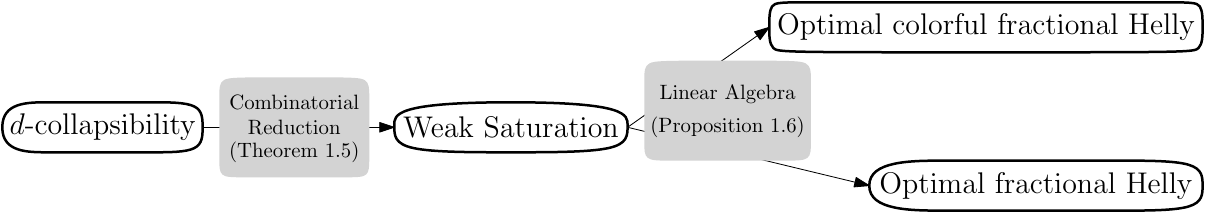}
    %\caption{Summary of the paper}
    %\label{fig:summary}
\end{figure}

\section{Reduction of (colorful) fractional Helly to weak saturation}

Let $K$ be a simplicial complex and $L$ be a complex whose faces are precisely the faces of $K$ of size at most $d-1$.
It was observed in \cite{AlonKalai} that $K$ is $d$-collapsible if and only if $L$ can be obtained from $K$ by a sequence of elementary $d$-collapses.
So, when a simplicial complex $K$ is $d$-collapsible, there is a sequence of simplicial complexes $K=K_0, K_1, \ldots, K_m = L$ where $K_{i+1}$ can be obtained from $K_i$ by an elementary $d$-collapse, that is, there is $\sigma_{i+1} \in K_i$ with $\size{\sigma_{i+1}}=d$ that is contained in unique maximal face $\tau_{i+1} \in K_i$ such that removing all faces $\eta$ with $\sigma_{i+1} \subseteq \eta \subseteq \tau_{i+1}$ from $K_i$ results in $K_{i+1}$. 
The sequence $(\sigma_1, \tau_1),(\sigma_2, \tau_2),\ldots,(\sigma_m, \tau_m)$ is called a {\em $d$-collapse sequence} of $K$.

\begin{proof}[{\bf Proof of Theorem~\ref{thm:weaksat}}]
Let $(\sigma_1, \tau_1), \ldots, (\sigma_m, \tau_m)$ be a $d$-collapse sequence of $K$.
Define $K_0=K$ and for every $j\in [m]$, let $K_j$ be the simplicial complex obtained after $j$ steps of the collapse. 
For $0 \leq j \leq m$, let $H_j$ be the $(d+1)$-graph on $V$ whose edges are all $d$-dimensional non-faces of $K_j$.

Since every face in $K$ has size at most $d+r$, we have $\size{\tau_j} \leq d+r$ for every $j\in [m]$.
On the other hand, since $\tau_j$ is the unique maximal face in $K_{j-1}$ containing $\sigma_j$, it must be that $\{v\} \cup \sigma_j \notin K_{j-1}$ for each $v \in V \setminus \tau_j$.
This implies that $\{v\} \cup \sigma_j \in H_{j-1}$ for every $v \in V \setminus \tau_j$.
Thus for every $w \in \tau_j \setminus \sigma_j$, adding the edge $\{w\} \cup \sigma_j$ to $H_{j-1}$ creates a copy of
$K^{(d+1)}_{1, \ldots, 1, n-d-r+1}$, 
namely $\{\{w\} \cup \sigma_j\} \cup \{\{v\} \cup \sigma_j : v \in V \setminus \tau_j\}$.
By the definition of $d$-collapse, $H_m$ is a copy of $K^{(d+1)}_n$ and $H_j = H_{j-1} \cup \{ \{w\} \cup \sigma_j : w \in \tau_j \setminus \sigma_j\}$ for every $j\in [m]$.
Therefore, any ordering of $E(H_m) \setminus E(H)$ where every $e \in E(H_i) \setminus E(H_{i-1})$ precedes every $e' \in E(H_j) \setminus E(H_{j-1})$ for all $1 \leq i < j \leq m$ is a weak $K^{(d+1)}_\bfm$-saturation sequence of~$H$.
\end{proof}

\begin{proof}[{\bf Proof of Theorem~\ref{thm:weaksatfamily}}]
Let $(\sigma_1, \tau_1), \ldots, (\sigma_m, \tau_m)$ be a $d$-collapse sequence of $K$.
Define $K_0=K$ and for every $j\in [m]$, let $K_j$ be the simplicial complex obtained after $j$ steps of the collapse. 
For $0\le j\le m$, let $H_j$ be the $(d+1)$-graph on $V = V_1 \cup \cdots \cup V_{d+1}$ whose edges are all rainbow non-faces of size $d+1$ in $K_j$.

Note that if $\sigma_j$ is not rainbow, then no rainbow simplices will be removed by the elementary $d$-collapse of $\sigma_j$ in the $d$-collapse sequence of $K$.
Suppose that $\sigma_j$ is a rainbow set for some $j\in [m]$.
Without loss of generality, we may assume that $\sigma_j \in K[V_2 \cup \cdots \cup V_{d+1}]$.
Consider $W_1 := \tau_j \cap V_1$.
Since every face in $K[V_1]$ has size at most $r_1$, we get $\size{W_1} \leq r_1$.
On the other hand, since $\tau_j$ is the unique maximal face in $K_{j-1}$ containing $\sigma_j$, for every $v \in V_1 \setminus W_1$ we know that $\{v\} \cup \sigma_j \notin K_{j-1}$.
Thus, whenever $E(H_j) \setminus E(H_{j-1})$ is nonempty, $W_1$ is nonempty and for every $w \in W_1$, adding $\{w\} \cup \sigma_j$ to $H_{j-1}$ creates a directed copy of $K^{(d+1)}_{\bfm_1}$, namely $\{\{w\} \cup \sigma_j\} \cup \{\{u\} \cup \sigma_j : u \in V_1 \setminus W_1\}$.
By the definition of $d$-collapse, $H_m$ is a copy of $K^{(d+1)}_\bfn$ with the vertex partition $V = V_1 \cup \cdots \cup V_{d+1}$.
Therefore, similar to before, any ordering of $E(H_m) \setminus E(H)$ where every $e \in E(H_i) \setminus E(H_{i-1})$ precedes every $e' \in E(H_j) \setminus E(H_{j-1})$ for all $1 \leq i < j \leq m$ is a directed weak $\{K^{(d+1)}_{\bfm_1}, \ldots, K^{(d+1)}_{\bfm_{d+1}}\}$-saturation sequence of $H$.
\end{proof}

\section{Weak saturation numbers}
We use a lemma from \cite{Bootstrap} that is helpful to deal with lower bounds in weak saturation problems. 
\begin{lemma}[{\cite[Lemma~3]{Bootstrap}}] 
\label{lemma}
Let $G$ be a hypergraph and $\cF$ be a family of hypergraphs. 
Let $X$ be the hypergraph on the edge set of $G$ such that $A \in E(X)$ if and only if $A$ forms a (directed) copy of some $F \in \cF$.
Let $W = \{w_e : e \in G\}$ be a set of vectors such that $w_e \in \spn( \{ w_f : f \in A \setminus\{e\} \} )$ for every edge $A \in E(X)$ and $e \in A$.
If $H$ is a (directed) weakly $\cF$-saturated hypergraph in $G$, 
then $\size{E(H)}\ge \dim(\spn(W))$.
In other words, $\wsat(G; \cF)$ (or, $\w(G; \cF)$) is at least $\dim(\spn(W))$.
\end{lemma}
\begin{proof}
Let $e_1, \ldots, e_m$ be a sequence of edges in $G$ that certifies that $H$ is (directed) weakly $\cF$-saturated.
Let $W_0 = \spn(\{w_e: e \in H\})$.
Having defined $W_{j-1}$, inductively define $W_j = \spn(W_{j-1} \cup \{w_{e_j}\})$ for $j\in [m]$.
By the assumption, we have $W_j=W_{j-1}$, and hence it must be $W_0=W_m$.
To finish the proof of the lemma, observe that
$\size{E(H)} \ge \dim(W_0) = \dim(W_m) = \dim(\spn(W))$.
\end{proof}

\begin{proof}[{\bf Proof of Proposition~\ref{thm:weaksatnumber}}]
Let $V$ be the vertex set of the host hypergraph~$K^{(d+1)}_n$.
Let $U \subseteq \bR[x_1, \ldots, x_{n-d-r}]$ be the vector space of homogeneous polynomials of degree $d+1$ that is spanned by the monomials $\{x_{j_1} x_{j_2} \cdots x_{j_{d+1}} : j_1, \dots, j_{d+1} \in [n-d-r]\}$. Note that $\dim(U)= \binom{n-r}{d+1}$.
For each vertex $v\in V$, assign a homogeneous polynomial of degree one $p(v)\in \bR[x_1,\dots,x_{n-d-r}]$ such that the tuples of corresponding coefficients are in general position in $\bR^{n-d-r}$. 
Then, for every $(d+1)$-tuple $(v_1,\dots,v_{d+1})$, assign the polynomial $\prod_{i\in [d+1]} p(v_i) \in U$.
With these vector assignments, one can easily check that Proposition~\ref{thm:weaksatnumber} follows from applying Lemma~\ref{lemma}.
\end{proof}

\begin{proof}[{\bf Proof of Proposition~\ref{thm:weaksatfamilynumber}}]
Let $V=V_1\cup \dots \cup V_{d+1}$ be the vertex partition of the host hypergraph~$K^{(d+1)}_\bfn$.
Let $U\subseteq \bR[x_{1,1},\dots,x_{1,n_1-r_1},\dots,x_{d+1,1},\dots,x_{d+1,n_{d+1}-r_{d+1}}]$ be the vector space spanned by the monomials $\{x_{1,j_1}x_{2,j_2}\cdots x_{d+1,j_{d+1}}: j_i\in [n_i-r_i]\}$, i.e., all degree ${d+1}$ rainbow monomials. Note that $\dim(U)=\prod_{i\in [d+1]} (n_i-r_i)$.
For each vertex $v\in V_i$ with $i\in [d+1]$, assign a homogeneous polynomial of degree one $p(v)\in \bR[x_{i,1},\dots,x_{i,n_i-r_i}]$ such that the tuples of corresponding coefficients are in general position in $\bR^{n_i-r_i}$. 
Then, for every rainbow $(d+1)$-tuple $(v_1,\dots,v_{d+1})$, assign the polynomial $\prod_{i\in [d+1]} p(v_i) \in U$.
Similar to before, Proposition~\ref{thm:weaksatfamilynumber} follows from applying Lemma~\ref{lemma}.
\end{proof}

\end{document}